\documentclass[12pt,reqno]{amsart}
\usepackage[a4paper]{geometry}
\usepackage{amsthm,hyperref,lmodern,mathrsfs,amssymb,amsmath,amsfonts}
\usepackage{esint}
\usepackage{enumerate}
\usepackage[utf8]{inputenc}
\usepackage{versions}

\usepackage{color}

\allowdisplaybreaks

\numberwithin{equation}{section}

\newcommand{\Hm}[1]{\leavevmode{\marginpar{\tiny%
$\hbox to 0mm{\hspace*{-0.5mm}$\leftarrow$\hss}%
\vcenter{\vrule depth 0.1mm height 0.1mm width \the\marginparwidth}%
\hbox to 0mm{\hss$\rightarrow$\hspace*{-0.5mm}}$\\\relax\raggedright
#1}}}

\newtheorem{thm}{Theorem}[section]
\newtheorem{lem}[thm]{Lemma}
\newtheorem{prop}[thm]{Proposition}
\newtheorem{cor}[thm]{Corollary}

\newtheorem{rmk}[thm]{Remark}

\def\R{{\mathbb{R}}}

\def\M{{\mathcal{M}\ \!\!}}
\def\D{{\mathcal{D}\ \!\!}}

\def\M{{\mathcal{M}\ \!\!}}

\def\Hess{{\mathrm{{\rm Hess }}}}

\def\and{{\mathrm{{\rm and}}}}

\def\fra{\mathfrak{a}}
\def\frb{\mathfrak{b}}

\def\o{{\omega\ \!\!}}

\def\Dtp{{\Delta_{\perp}}}
\def\Dtf{{\overset{\rightarrow}{\Delta}}}

\def\beq{\begin{equation}}
\def\eeq{\end{equation}}

\begin{document}

\title[The spectral gap]{ Lower bounds for the  spectral gap and an extension of the Bonnet-Myers theorem}

\author{Michel~Bonnefont and El Maati~Ouhabaz} \address{ Institut de Math\'ematiques de Bordeaux, Universit\'e Bordeaux, UMR CNRS 5251, 351 Cours de la Lib\'eration 33405, Talence.  France} \thanks{The authors are  partially supported by the ANR project RAGE: ANR-18-CE-0012-01.\\
M.B. is also partially supported by the ANR project QuAMProcs: ANR-19-CE40-0010.} 
\email{michel.bonnefont@math.u-bordeaux.fr}
\email{Elmaati.Ouhabaz@math.u-bordeaux.fr}




\begin{abstract} {On a fairly general class of  Riemannian manifolds ${\mathcal M}$,  we prove lower  estimates in terms of the Ricci curvature for the spectral bound (when  $\mathcal M$ has infinite volume) and for the spectral gap (when ${\mathcal M}$ has finite volume) for the Laplace-Beltrami operator. As a byproduct of our results we obtain an extension of the Bonnet-Myers theorem on  the compactness of the manifold. We also prove lower bounds for the spectral gap for Ornstein-Uhlenbeck type operators on weighted manifolds. As an application we prove lower bounds for the spectral gap of some probability measures  $e^{-V(x)} dx$ with a potential $V$ which is radial only inside or outside a certain ball of $\R^n$.}
 \end{abstract}
 
\vspace{1cm}

\maketitle



 \vspace{1cm}
 \tableofcontents

\newpage

\section{Introduction and the main results}\label{sec1}

The primary aim of the present paper is to  understand the link between positivity of the spectral bound (respectively, the spectral gap) of the Laplace-Beltrami operator for manifolds with infinite volume (respectively, for manifolds with finite volume) and the Ricci curvature. This is a subject that has been investigated by many authors. The results in the literature often use lower bounds of the Ricci curvature. In the  present paper, we will see that a mean condition on the smallest eigenvalue of the Ricci tensor can be used in an efficient way in order to treat such questions. 

Let $\mathcal M$ be a complete Riemannian manifold. We denote by $B(x,r)$, for $x \in {\mathcal M}$ and $r > 0$, the open ball of centre $x$ and radius $r$. As usual, $\Delta$ denotes the non-negative Laplace-Beltrami operator and $\nabla$ the gradient. The main two properties we need for ${\mathcal M}$ are the standard covering by balls with finite overlapping together with the $L^2$-Poincar\'e inequality on balls.  More precisely,\\

$(A1)$ {\it The covering property.} For any $R >0$, there exist a sequence $(x_k) \in {\mathcal M}$ and a number $N_R$ such that 
${\mathcal M} = \cup_k B(x_k, R)$ and each $x \in {\mathcal M}$ is contained in at most $N_R$ balls $B(x_k,R)$.\\

$(A2)$ {\it Poincar\'e on balls.} For every $R > 0$, there exists a constant $\kappa_R$ such that  
\begin{equation}\label{poincare}
  \kappa_R  \int_B | u - u_B |^2\, dx \le \int_B |\nabla u |^2\, dx
\end{equation}
for every ball $B = B(y,R)$ and $u \in H^1(B)$. Here $u_B := \frac{1}{vol(B)} \int_B u\, dx := \fint_B u\, dx$ is the mean of $u$ on $B$.\\

It is a well known fact that $(A1)$ and $(A2)$ are both satisfied if the manifold has Ricci curvature bounded from below by $-k^2$. In this case, the Poincar\'e inequality \eqref{poincare} holds with $\kappa_R = Ce^{-k^2 R} R^{-2}$ for all $R > 0$.  See Buser \cite{Buser} or  Saloff-Coste \cite{Saloff}, Theorem 5.6.5. 


We denote by $s(\Delta)$ the spectral bound of $\Delta$, that is
$$ s(\Delta) = \inf \sigma(\Delta),$$
the infimum of its spectrum $\sigma(\Delta)$. If $\M$ is compact then 
$\sigma(\Delta)$ consists of a discrete set of eigenvalues $ (\lambda_k)_{k\ge0}$.  In addition,  $\lambda_0 = 0$ and $\lambda_1 > 0$. 
For general manifolds with $s(\Delta) = 0$, one defines 
$$ \lambda_1 = \inf  \left[ \sigma(\Delta)\setminus \{0\}\right].$$
If $\lambda_1 > 0$, one says that $\Delta$ (or $\M$) has a spectral gap (and $\lambda_1$ is the value of the spectral gap). It is of great interest to decide whether a Riemannian manifold $\M$ has a spectral gap and it is more interesting to have qualitative lower bounds for $\lambda_1$.
The spectral gap gives the rate of convergence of the process towards its stationary state. Hence it is of interest in probability theory and it is also related to several functional inequalities. See Ledoux \cite{Ledoux1} or the monograph by Bakry, Gentil and Ledoux \cite{Bakry-Gentil-Ledoux} for an overview and historical notes. 

 We prove lower bounds for $s(\Delta)$ and for the spectral gap.  Our bounds are given  in terms of the smallest eigenvalue
$\rho(x)$ of the Ricci tensor  $Ricci(x)$ at $x \in \M$. The function $\rho$ can of course be written as $\rho^+ - \rho^-$ with $\rho^+$ and $\rho^-$ are its positive and negative parts, respectively. We introduce the following average condition on $\rho^+$ and Hardy type inequality on $\rho^-$. 
\begin{equation}\label{meanR}
\delta(R) := \inf_{ p \in \M}  \fint_{B(p,R)} \rho ^+(x) dx  > 0.
\end{equation}
\begin{equation}\label{formbdd}
\int_{\M} \rho ^{-} | u |^2 \le \alpha \left[ \int_{\M} | \nabla u |^2 + \int_{\M} \rho ^+ | u |^2 \right] \quad \forall u \in H^1(\M).
\end{equation} 

We give in the following proposition an overview of our bounds. More general  statements  will be given in the coming sections. 

\begin{prop} \label{prop1} 
Let $\M$ be a complete Riemannian manifold. Suppose  that  $\rho ^-$  is form bounded in the sense of \eqref{formbdd} for some $\alpha \in [0, 1)$. 
\begin{itemize}
\item If $vol({\mathcal M}) = \infty$, then $s(\Delta) \ge s(\Delta + \rho)$. In particular, if $\rho \ge 0$ then $s(\Delta) = s(\Delta + \rho )$.
\item If  $vol({\mathcal M}) < \infty$, then  $s(\Delta) = 0$ and $\lambda_1(\Delta) \ge s(\Delta + \rho)$. 
\end{itemize}
Suppose now that $\M$ satisfies the covering property $(A1)$ and the Poincar\'e inequality $(A2)$. Suppose that  $\rho ^+$  is bounded and satisfies the average condition \eqref{meanR} below for some $R > 0$.  
\begin{itemize} 
\item  If $vol(\M) = \infty$, then 
\[
s(\Delta) \ge \, \sup_{R > 0}\,    \frac{1-\alpha}{N_R} \frac{ \kappa_R}{\kappa_R + 2 \| \rho^+ \|_{L^\infty(\M)}} \delta(R)
\]
 where $N_R$ and $\kappa_R$ are the constants  in $(A1)$ and $(A2)$. 
  \item If 
$vol(\M) < \infty$, then $\Delta$ has a spectral gap with the lower bound
\[
\lambda_1(\Delta) \ge \, \sup_{R > 0}\,    \frac{1-\alpha}{N_R} \frac{\kappa_R}{\kappa_R + 2  \| \rho^+ \|_{L^\infty(\M)}} \delta(R). 
\]
\end{itemize}
\end{prop}


Let us recall some related results. It is known that if the volume grows exponentially in every direction then $s(\Delta) > 0$. This exponential growth holds if the sectional curvature is negative  at infinity. See Sturm \cite{Sturm}, Proposition 2 and the discussion on p. 445. Concerning the spectral gap, the literature on this subject is extensive and we quote only few papers and refer to the references there. By the well known Lichnerovicz' lower bound for a manifold of dimension $n$  satisfying $Ricci \ge k > 0$, one has 
$\lambda_1 \ge \frac{n}{n-1}k$. Li-Yau's lower bound, improved by Zhong and Yang \cite{ZY84} for compact manifolds with non-negative Ricci curvature and diameter $d$, provides  $\lambda_1 \ge \frac{\pi^2}{d^2}$.  Note also that  Lichnerovicz lower bound can be obtained via the Bakry-\'Emery criterion in a general setting. See the curvature-dimension criterion  $CD(k,n)$ in \cite{Bakry-Gentil-Ledoux}, Theorem 4.8.4.\\
 For manifolds with Ricci curvature that can have a non-trivial negative part, Aubry \cite{aubry} proved a lower bound in terms of the dimension $n$, $vol(M)$ and the quantity on the left-hand side of \eqref{Aubry} below. In \cite{Carron-Rose} Carron and Rose proved some lower bounds for $\lambda_1$ using Kato type assumptions on $\rho$ or some means of perturbations of $\rho^-$. Our lower bounds in Proposition \ref{prop1} (see also Theorem \ref{thm-spectral}) are of different nature and are expressed in terms of the mean of $\rho$ on balls of the manifold. 

Let us now give some ideas of proof of our lower bounds. Let $\Dtp$ be either $\Delta$ if $vol(M) = \infty$ or $\Delta(I-\Pi)$ with $\Pi$ being the projection onto the null set of $\Delta$ if $vol(M) < \infty$. Next, set $\Dtf = d d^* + d^*d$ be the Hodge-Laplacian on 
$1$-differential forms and $\Dtf_{\overline{d}} = dd^*$ its realization on the closure of exact forms. We shall see that one has the intertwining  property
\[ e^{-t \Dtf_{\overline{d}}} = R\, e^{-t\Dtp}\, R^*
\]
where $R$ is the Riesz transform $d \Dtp^{-\frac{1}{2}}$. One obtains from this  $\sigma(\Dtf_{\overline{d}}) = \sigma(\Dtp)$. This equality of the spectra may be  known to experts (at least in some specific situations) but we could not find a reference where it is written explicitly. Thus, we  give the full details both for the reader's convenience and also because we prove such equality for any complete manifold. A particular case of the above equality is that  $s(\Dtp) = s(\Dtf_{\overline{d}})$. On the other hand, by the standard minimisation of the spectral bound, it is clear that $s(\Dtf_{\overline{d}}) \ge s(\Dtf)$. On applying the pointwise domination
\begin{equation}\label{0-dom}
 | e^{-t\Dtf} \o | \le e^{-t(\Delta + \rho)} | \o | 
 \end{equation}
it follows that $s(\Dtp) \ge s(\Delta + \rho)$. This is the inequality stated in the first two assertions of Proposition \ref{prop1}. One of the gains with this inequality is that any lower estimate for the spectral bound of the Schr\"odinger operator $\Delta + \rho$ is also a lower bound for $s(\Dtp)$. For Schr\"odinger operators on Riemannian manifolds, such lower estimate where proved by Ouhabaz \cite{Ouhabaz01} in terms of the mean $\delta(R)$. In Section \ref{sec3} we consider again this question, we give some historical information and produce lower bounds which appear in the last two assertions of Proposition \ref{prop1}.  


Recall that the Bonnet-Myers theorem asserts that if the Ricci curvature is uniformly bounded from below by a positive constant then the manifold is compact. We prove  the following extensions. We refer to  Section \ref{sec5} for some additional  comments. We denote by $n$ the dimension of the manifold $\M$. 

\begin{thm}\label{thm2}
Suppose that the manifold  $\M$  satisfies $(A1)$ and  $(A2)$ and  the volume of balls has a sub-exponential  growth. Suppose that $\inf \{ \rho ^+, w_0\}$ satisfies the average condition \eqref{meanR} for some $R > 0$ and some constant $w_0>0$ and that $\rho ^-$ satisfies \eqref{formbdd} for some $\alpha \in [0,1)$. 
Then $\M$ has finite volume. In addition, if for some $x_{0} \in \M$,  
\begin{equation}\label{Gromov}
 Ricci(x) \ge -\frac{n}{n-1} \frac{1}{d(x,x_0)^2} 
 \end{equation}
when  $d(x,x_{0})$ is large enough, then $\M$ is compact.  
\end{thm}

Recall that the volume of balls $B(x,r)$ has sub-exponential growth if for some $x_0 \in \M$
\begin{equation}\label{vol-exp}
\limsup_{r \to \infty} \frac{1}{r} \log [vol(B(x_0,r))] = 0.
\end{equation}
Note that if $\rho ^-(x)$ converges to $0$ when the distance  $d(x,x_0)$ (from a fixed point $x_0$ of $\M$) tends  to infinity, then the volume 
has sub-exponential growth (cf. \cite{Sturm}).  

For manifolds whose Ricci curvature is asymptotically non-negative we can remove the assumption  \eqref{formbdd} on the negative part of $\rho$. Before giving the statement, let us say that the  Ricci curvature $Ricci$ is asymptotically  non-negative  if for every $\epsilon >0$ there exists a compact $K \subseteq \M$ such that 
\begin{equation}\label{riccasympt}
\rho^-(x) < \epsilon  \quad{\rm for }\ x \in \M\setminus K.
\end{equation} 

\begin{thm}\label{thm1.6} Let $\M$ be a complete manifold with Ricci curvature asymptotically non-negative. Suppose that $\inf \{ \rho ^+, w_0\}$ satisfies the average condition \eqref{meanR} for some $R > 0$ and some constant $w_0>0$. Then $\M$ has finite volume. If in addition \eqref{Gromov} is satisfied for some $x_{0} \in \M$,  
 then $\M$ is compact.
 \end{thm}

Clearly, if \eqref{Gromov} holds for all $x \in \M$ with $d(x, x_0) \ge R$ then \eqref{riccasympt} holds. Thus,  the theorem shows  that   \eqref{Gromov} for $d(x, x_0) \ge R$ and  the average condition \eqref{meanR} for $\inf \{ \rho ^+, w_0\}$  imply that $\M$ is compact.

In  the previous two  theorems we  allow the Ricci curvature to have a non-trivial  negative part. We also allow the positive part of the  Ricci curvature to vanish  on some subsets of $\M$. The Bonnet-Myers theorem does not allow this since it requires  $\rho (x) \ge \nu$ for some  $\nu > 0$ and all $x \in \M$. The Bonnet-Myers theorem was extended by many authors. The literature on this subject is quite extensive and we mention only few papers which deal with conditions in the sprit of ours. Ambrose \cite{Ambrose} and later Mastrolia, Rimoldi and Veronelli \cite{MRV} assume an integrability condition for the Ricci tensor along geodesics. We also refer to Cheeger, Gromov and Taylor \cite{CGT} where a quadratic lower bound for the Ricci curvature is assumed. Aubry \cite{aubry} proved that if
\begin{equation}\label{Aubry}
 \int_\M \left[ (\rho- (n-1))^- \right]^p\, dx < \infty
\end{equation}
for some $p > \frac{n}{2}$, then $\M$ has finite volume. Under a slightly different condition, he proved among other things, the compactness of the manifold and a diameter estimate. A related condition to \eqref{Aubry}, with the integral taken on balls,  was already considered by Petersen and Sprouse \cite{Petersen-Sprouse} who also proved compactness and diameter estimates for $\M$. More recently, Rose \cite{Rose} proved that $\M$ has finite volume provided the $L^\infty$-condition 
\begin{equation}\label{rose}
 \int_0^T  \| e^{-t\Delta} (\rho- k)^- \|_\infty \, dt \le 1 - e^{-kT/4}
\end{equation}
holds for some positive $T$ and $k$. He also proved diameter estimate under a Kato type condition by  relying on Carron and Rose \cite{Carron-Rose} who obtain diameter estimates for compact manifold using positivity of certain  Schr\"odinger operators with potential involving $\rho$ and some constants.  

Our condition on $\rho$ in Theorem \ref{thm2} is not in the same spirit as the conditions in the previous works. In order to check \eqref{rose} one needs some decay for the $L^\infty$-norm of the heat semigroup associated with a Schr\"odinger operator.  Such $L^\infty$-decay is in general more complicated to check than the $L^2$-decay as we do in the present paper. Next, if we consider the easiest case where $\rho \ge 0$, then our mean condition  does not a priori assume the set
$\{ \rho = 0\}$ to have finite measure whereas 
 \eqref{Aubry} readily implies this set to have finite measure. 

Our proof for Theorem \ref{thm2} is  rather simple and it is based on our results on the (strict) positivity of the spectral bound. Note however that we do not  prove diameter estimates for $\M$. A reasonable diameter estimate which uses our mean condition \eqref{meanR} remains unclear to us.

Most of our results remain valid in the setting of weighted manifolds. That is, we consider on $\M$ a measure $\mu = e^{-V} dx$ instead of the Riemannian measure $dx$. This gives rise to  a self-adjoint Ornstein-Uhlenbeck type operator $\Delta_\mu$ and the Hodge-Laplacian
$\Dtf_\mu = d d_\mu^* + d_\mu^*d$. In this setting,  Bochner's formula reads as $\Dtf_\mu = \tilde{\Delta_\mu} + {\mathcal R}$ with 
${\mathcal R} = Ricci + Hess V$, where $Hess V$ is the Hessian of $V$.  The domination property \eqref{0-dom} holds with  $\rho(x)$ being  the smallest eigenvalue of the new tensor ${\mathcal R}$.   
We can then  repeat all our arguments by replacing $Ricci$ by the tensor ${\mathcal R}$.  We obtain the same  lower bounds for the spectral gap as before in terms of the smallest eigenvalue of $Ricci(x) + Hess V(x)$.  In the  last section of this paper we consider perturbations of the radial measure $\mu_\alpha = e^{-\frac{|x|^\alpha}{\alpha}}\, dx$  in the sense that we do not require for our measures to be radial inside or outside a given ball depending whether $\alpha \in (1,2]$ or $\alpha \in [2, \infty)$.  We prove lower bounds for the corresponding spectral gap which  have the optimal asymptotics with respect to the dimension. Lower and upper bounds for the spectral gap of $\mu_\alpha$ were given by Bobkov \cite{Bobkov} and by Bonnefont, Joulin and Ma \cite{Bonnefont-Joulin-Ma}. 

\medskip
\noindent{\bf Acknowledgements.}  We wish to thank the referee for his/her comments.  Theorem \ref{thm1.6} was suggested by him/her and it  improves our previous result where we considered only manifolds with non-negative Ricci curvature.

\section{Setting and preliminary tools}\label{sec2}

In this section we define the spaces and operators which will be needed throughout this paper.

We consider a complete Riemannian manifold $(\M,g)$ and denote by $dx$ its volume measure. We denote by $\nabla$ its gradient, $d$ the exterior derivative on differential forms and $d^*$ its formal adjoint. 

By $L^2 = L^2(\M)$ we denote the usual $L^2$-space of functions on $\M$ with respect to $dx$. The corresponding $L^2$ space of differential forms of order $k \ge 1$ is $L^2(\Lambda^k)$, endowed with the Riemann measure $dx$.  We define 
\[ L^2_{\overline{d}}(\Lambda^1) = \overline{ \{ df, f \in H^1(\M)\} }^{L^2(\Lambda^1)},\]
the closure in $L^2(\Lambda^1)$ of the space $\{ df, f \in H^1(\M)\}$ where $H^1(\M)$ stands for the standard Sobolev space. 
The scalar products in $L^2$ or $L^2(\Lambda^k)$ are both denoted by $\langle \cdot, \cdot \rangle$. Hence for $f, g \in L^2$ and $\omega, \eta \in L^2(\Lambda^1)$
\[
\langle f, g \rangle = \int_\M f(x)g(x)\, dx \, \, \, {\rm and}\, \, \, \langle \omega, \eta \rangle = \int_{\M} \omega(x).\eta(x) \, dx
\]
and notice that $\omega(x).\eta(x)$ is the scalar product in the cotangent space $T^*_x\M$. It should be pointed out that this latter  scalar product depends on the point $x \in \M$ but we do not make this explicit in order to keep the notation as simple as possible. The corresponding norm in $L^2(\Lambda^1)$ (and of course in $ L^2_{\overline{d}} (\Lambda^1)$) is
$$ \| \omega \|_2  := \left( \int_{\M} | \omega(x) |_x^2 \, dx \right)^{1/2}$$
where $| \cdot |_x$ is the norm in $T^*_x\M$. For simplicity, the norms in the above $L^2$-spaces (on functions or on forms) are all denoted by $\| \cdot  \|_2$. 

The operators which will be used in this paper  are all self-adjoint in appropriate Hilbert spaces. The most convenient way to define them is via their quadratic forms. \\

\noindent\underline{{\it The Laplace-Beltrami operators}}. We define on $L^2$ the symmetric form
\[ \fra(u,v) = \int_{\M} \nabla u.\nabla v\, dx, \  u, v \in H^1(\M).\]
The corresponding operator is the nonnegative Laplacian $\Delta$. \\
Note that for $u \in H^1(\M)$,  $\Delta u = 0$ {\it if and only if } $u$ is constant. Therefore, $0$ is an eigenvalue of $\Delta$ (and hence the corresponding eigenfunctions are the constants)  {\it if and only if } $\M$ has finite volume.  We define the orthogonal space to $\ker(\Delta)$ 
\[ L^2_\perp := \{ u \in L^2: \int_{\M} u\, dx = 0 \}\]
 for $vol(\M) < \infty$  and  $L^2_\perp = L^2$ if $vol(\M) = \infty$. If $\Pi$ denotes the projection of $L^2$ onto $L^2_\perp$ then the part of 
$\Delta$ on $L^2_\perp$ is $\Delta(I-\Pi) = (I-\Pi) \Delta (I-\Pi)$. It is a self-adjoint operator on $L^2_\perp$. In order to treat simultaneously the cases of finite and infinite volume, we set
\[
\Dtp:= \left\{ \begin{array}{c c}
               \Delta & \textrm{ if the volume is infinite}\\
               \Delta (I-\Pi) & \textrm{ if the volume is finite.}
               \end{array}\right.\]
It can be seen readily that
\[ D(\Delta)  = \{ f \in H^1(\M), \, d^* d f \in L^2(\M) \} \ {\rm and} \ \Delta f = d^* d f.
\]
In the case of finite volume it is elementary to prove that the semigroup $e^{-t\Delta}$ leaves invariant $L^2_\perp$ and its restriction to this space is $e^{-t \Dtp}$. \\

\noindent\underline{{\it Schr\"odinger operators}}. Let  $0 \le V \in L^1_{loc}(\M)$. Then  one defines the Schr\"odinger operator $\Delta + V$ on $L^2$ as the operator associated with the form
\[ \frb(u,v) = \int_{\M} \nabla u. \nabla v\, dx + \int_{\M} V u v \, dx \]
with domain 
\[ \{ u \in H^1(\M): \int_{\M} V |u|^2\, dx < \infty \}. \]
In the case where $V = V^+ - V^-$ has a non-trivial negative part $V^-$, then  one usually assumes  that for some $\alpha \in [0, 1)$ and some constant $c_\alpha$
\begin{equation}\label{eq-form-moins}
 \int_{\M} V^- |u|^2\, dx \le \alpha \left[ \int_{\M} \nabla u. \nabla v\, dx + \int_{\M} V^+ u v \, dx \right] + c_\alpha \int_{\M} |u|^2\, dx 
 \end{equation}
for all $u \in H^1(\M)$ such that $ \int_{\M} V^+ |u|^2\, dx < \infty$. The corresponding operator $\Delta + V$ to the form $\frb$ is self-adjoint and bounded from below (by $-c_\alpha$).  \\

\noindent\underline{{\it The Hodge Laplacians}}. There are two Laplacians on differential forms which will be of interest to us. The first one is the usual Hodge Laplacian $\Dtf$ on $L^2(\Lambda^1)$. It is the  non-negative self-adjoint  operator associated with the form
\[
\overset{\rightarrow}{\fra}(\omega, \eta) = \int_{\M} d \omega(x). d \eta(x) \, dx + \int_{\M} d^* \omega(x) d^* \eta(x) \, dx 
\]
with domain the Sobolev space
\[
H^1(\Lambda^1) = \{ \omega \in L^2(\Lambda^1): d\omega \in L^2(\Lambda^2) \ {\rm and} \ d^*\omega \in L^2 \}.
\]
If one defines the Sobolev type space 
\begin{equation}\label{W-sp}
{\mathcal W} := \{ \omega \in H^1(\Lambda^1):  d d^* \omega,  d^* d \omega \in L^2(\Lambda^1) \},
\end{equation}
then  ${\mathcal W} $ is contained in the domain of $\Dtf$ and on this space
the operator $\Dtf$ is given explicitly by the expression $\Dtf = d d^* + d^* d$.

Next, we introduce the operator $\Dtf_{\overline{d}}$ on $L^2_{\overline{d}}(\Lambda^1)$ as the  operator associated to the same form
\[
\overset{\rightarrow}{\fra_1}(\omega, \eta) =  \overset{\rightarrow}{\fra}(\omega, \eta) 
\]
but now with domain 
\[
\D(\overset{\rightarrow}{\fra_1}) = H^1(\Lambda^1)\cap  L^2_{\overline{d}}(\Lambda^1).
\]
It is easy to see that this space, endowed with the norm $\| \omega \|_2 + \| d \omega \|_2 + \| d^* \omega \|_2$,  is complete so that the form $\overset{\rightarrow}{\fra_1}$ is closed. Hence $\Dtf_{\overline{d}} $ is well defined and it is a self-adjoint operator on $L^2_{\overline{d}}(\Lambda^1)$. Moreover, since for  
$\omega \in  L^2_{\overline{d}}(\Lambda^1)$, $\omega = \lim_n df_n$ in the $L^2(\Lambda^1)$-sense, one has 
$$\langle  d \omega, \eta \rangle = \langle  \omega, d^*\eta \rangle =  \lim_n \langle  f_n,  d^* d^*\eta \rangle = 0$$
for all $\eta \in C_c^\infty(\Lambda^2)$ and so
\begin{equation}\label{dw=0}
 d \omega = 0 \, \, \, {\rm for \ all }\, \,  \omega \in  L^2_{\overline{d}}(\Lambda^1).
 \end{equation}
 Thus, the form $\overset{\rightarrow}{\fra_1}$ has the expression  
 \[
\overset{\rightarrow}{\fra_1}(\omega, \eta) =  \int_{\M} d^* \omega(x) d^* \eta(x) \, dx 
\]
 on its domain $\D(\overset{\rightarrow}{\fra_1})$. Since  $L^2_{\overline{d}}(\Lambda^1)$ is a closed subspace of $L^2(\Lambda^1)$ and  $H^1(\Lambda^1) \cap L^2_{\overline{d}}(\Lambda^1)$ is dense in  $L^2_{\overline{d}}(\Lambda^1)$, one has easily
\[
 \D(\Dtf) \cap L^2_{\overline{d}}(\Lambda^1) \subset \D(\Dtf_{\overline{d}})
\]
and for $\omega \in  \D(\Dtf) \cap L^2_{\overline{d}}(\Lambda^1)$, $\Dtf_{\overline{d}} \omega $ coincides with the orthogonal projection (of $L^2(\Lambda^1)$ onto  $L^2_{\overline{d}}(\Lambda^1)$) of $\Dtf \omega$.
In particular, ${\mathcal W}  \cap L^2_{\overline{d}}(\Lambda^1)$ is contained  in the domain of $\Dtf_{\overline{d}}$ and  the operator $\Dtf_{\overline{d}}$ is given by the expression $\Dtf_{\overline{d}} = d d^* $ on this domain. The reason is that $d^* d \omega = 0$ by \eqref{dw=0} and $d^* \omega \in H^1(\M)$ for 
$\omega \in {\mathcal W}$. 
To sum up, we have ${\mathcal W}  \cap L^2_{\overline{d}}(\Lambda^1) \subset {\mathcal W}  $ and 
\begin{equation}\label{eq:egaliteDtf}
\Dtf \omega= (d d^* + d^* d) \omega= d d^* \omega = \Dtf_{\overline{d}} \omega\  {\rm for}\  \omega \in {\mathcal W}  \cap L^2_{\overline{d}}(\Lambda^1).
\end{equation}

It is interesting to point out that the semigroup $e^{-t\Dtf}$ leaves invariant $L^2_{\overline{d}}(\Lambda^1)$ and its restriction to this space is $e^{-t \Dtf_{\overline{d}}}$. 
\begin{prop}\label{prop2.1.0} For every $\omega \in  L^2_{\overline{d}}(\Lambda^1)$,
$$e^{- t \Dtf} \omega = e^{-t \Dtf_{\overline{d}}}\omega.$$
\end{prop}
\begin{proof} It is enough to prove the equality for the resolvents, i.e.,
\begin{equation}\label{eqR}
 (I + \Dtf)^{-1} \omega = (I + \Dtf_{\overline{d}})^{-1} \omega.
 \end{equation}
Indeed, this would imply by iteration  
$$(I + \frac{t}{n}\Dtf)^{-n} \omega = (I + \frac{t}{n}\Dtf_{\overline{d}})^{-n} \omega$$
and letting $n \to \infty$, the equality holds then for the semigroups. 

Since the resolvents are bounded operators, it suffices to prove \eqref{eqR} for $\omega = df$ with $f \in H^1(\M)$. Now, since (see \eqref{eq.resol} below)
\[
 (I + \Dtf)^{-1} df = d (I + \Delta)^{-1} f,
\]

one has $(I + \Dtf)^{-1} df \in  L^2_{\overline{d}}(\Lambda^1) \cap  \D(\Dtf) \subset  \D(\Dtf_{\overline{d}})$.

 In particular, one can apply $\Dtf_{\overline{d}}$ to $(I + \Dtf)^{-1} df$. On the other hand, for $\omega_0 := (I + \Dtf)^{-1} df = d (I+ \Delta)^{-1} f$, we have $\omega_0 \in {\mathcal W}$  (see \eqref{W-sp}). This follows from the facts that  $d^*d \omega_0 = 0 \in L^2(\Lambda^1)$ and $dd^* \omega_0 = d \Delta (I+ \Delta)^{-1}f = d [ f - (I+ \Delta)^{-1}f] \in L^2(\Lambda^1)$ because $f \in H^1(\M)$. We apply \eqref{eq:egaliteDtf} to $\omega_0$ and obtain
 $$ (I + \Dtf_{\overline{d}}) (I + \Dtf)^{-1} df = (I + \Dtf) (I + \Dtf)^{-1} df = df.$$
 This proves \eqref{eqR}. \end{proof}

The commutation property 
$$ d\Delta = d \Delta_\perp = \Dtf d$$
follows formally from the definitions of the operators. However, some domain questions are hidden here and one has to be careful on which functions both terms do act as operators. 
The commutation with the semigroup does not have this domain task. 
\begin{prop}\label{prop2.1}
For every $t > 0$ and every function $f \in H^1(\M)$,
\beq\label{eq:d-e-tD}
d  e^{-t \Dtp } f = e^{- t \Dtf}\, df = e^{-t \Dtf_{\overline{d}}}\, df.
\eeq
\end{prop}
\begin{proof}
The second equality is already stated in the previous proposition. We prove the first one. 

The operators $I + \Delta$ or $I + \Dtp$  are inversible. In addition, $ (I+ \Dtp)^{-1}$ is bounded on $L^2_\perp$ and has values in $H^1(\M)$ (or even in $\D(\Delta_\perp)$).
Now  for $f\in \D(\Dtp)$ with  $\Delta f  = \Dtp f \in H^1(\M)$ one has $d f \in {\mathcal W}$ and 
\[
d (I+ \Dtp) f= (I+  \Dtf) df.
\]
Let  $f=(I+\Dtp)^{-1} g$ for $g\in H^1(\M)$ ($g\perp 1$ if the measure is finite).
Then   $(I+\Dtp)^{-1} g \in \D(\Dtp) $ and 
\[
\Dtp (I+\Dtp)^{-1} g = g - (I+\Dtp)^{-1} g \in H^1(\M).
\]
Thus we can apply the above commutation property with such $f$ and we obtain
\[
dg= (I+  \Dtf)d (I+\Dtp)^{-1} g.
\]
Applying the bounded operator $(I+\Dtf)^{-1}$ to both sides it follows that 
\beq\label{eq.resol}
(I+\Dtf)^{-1} dg=d (I+\Dtp)^{-1} g.
\eeq
The same relation holds for $(I+ \alpha \Dtp)$ for any $\alpha>0$.
Iterating this relation one sees also that it holds for $(I+ \alpha \Dtp)^{-k}$.
In particular, for $g\in H^1(\M)$,
 \[
\left(I+\frac{t}{n}\Dtf\right)^{-n} dg=d \left(I+\frac{t}{n}\Dtp\right)^{-n} g.
\]
The left side converges in $L^2(\Lambda^1)$ towards $e^{-t\Dtf}dg$. On the other hand, $\left(I+\frac{t}{n}\Dtp\right)^{-n} g$ converges 
in $L^2$ to $e^{-t\Dtp}g$ and $d \left(I+\frac{t}{n}\Dtp\right)^{-n} g$ is a Cauchy sequence. This implies that $\left(I+\frac{t}{n}\Dtp\right)^{-n} g$ is a Cauchy sequence in $H^1(\M)$. It follows  that $d \left(I+\frac{t}{n}\Dtp\right)^{-n} g$ converges to $d e^{-t\Dtp}g$. This proves that 
 for all $g\in H^1(\M)$,
\[d  e^{-t \Dtp } g = e^{- t \Dtf} dg.
\]
\end{proof}

Let $\rho(x)$ be the smallest eigenvalue of the tensor $Ricci(x)$ at the point $x \in \M$ and suppose that  its negative part $\rho ^-$ is form bounded in the sense of \eqref{eq-form-moins} for some $\alpha \in [0, 1)$. Thus, we can define the Schr\"odinger operator $\Delta + \rho$. 
One of the remarkable facts on the semigroup $e^{-t\Dtf}$ is the very well known pointwise domination
\begin{equation}\label{Hodge-domination}
| e^{-t\Dtf} \omega (x) |_x \le e^{-t(\Delta + \rho)} |\omega(x)|_x
\end{equation}
for all $t > 0$ and and $\omega \in L^2(\Lambda^1)$. Here, as we mentioned above, $ | \cdot |_x$ denotes the norm in the cotangent space 
$T_x^*M$. This domination holds also for the semigroup $e^{-t\Dtf_{\overline{d}}}$ due to Proposition \ref{prop2.1.0}. The classical way to prove \eqref{Hodge-domination} is to write by the Bochner's formula $\Dtf = \tilde{\Delta} + Ricci$ where  
$\tilde{\Delta} = \nabla^* \nabla $ is the rough Laplacian (here $\nabla$ is the Levi-Civita connexion). By Kato's inequality, one proves (see \cite{HSU})
\begin{equation}\label{Hodge-domination-1}
| e^{-t\tilde{\Delta} }\omega (x) |_x \le e^{-t\Delta} |\omega(x)|_x.
\end{equation}
From this and characterization of domination of semigroups \cite{Ouhabaz99} or Theorem 2.30 in \cite{Ouhabaz05},  the domain of the quadratic form of 
$\tilde{\Delta}$ is an ideal of $H^1(\M)$ and one has (Kato's inequality)
\begin{equation}\label{Kato-1}
\int_\M \nabla \omega. \nabla \eta\, dx \ge \int_\M \nabla |\omega |. \nabla | \eta|\, dx
\end{equation}
for differential forms $\omega$ and $\eta$ such that $\omega(x). \eta(x) = |\omega(x)|_x |\eta(x)|_x$. Note that the term on the left hand side is the scalar product in the cotangent space $T_x^*\M$ and thus this equality means that $\omega(x) = \lambda(x) \eta(x)$ with $\lambda(x) > 0$. Let now $\mathcal R$ be a symmetric tensor and $V(x)$ the smallest eigenvalue of ${\mathcal R}(x)$ such that $V^+ \in L^1_{loc}(\M)$ and $V^-$ is form bounded in the sense of  \eqref{eq-form-moins}.  For $\omega $ and $\eta$ as above one has
\[
{\mathcal R}(x) \omega(x). \eta(x) = \lambda(x) {\mathcal R}(x) \eta(x). \eta(x) \ge \lambda(x) V(x)  |\eta(x)|_x^2 = V(x) |\omega(x)|_x | \eta(x)|_x.
\]
This and \eqref{Kato-1} imply that 
\begin{eqnarray*}
&& \hspace{-2cm} \int_\M \nabla \omega. \nabla \eta\, dx  + \int_M  {\mathcal R}(x) \omega(x).\eta(x)\, dx\\
&& \hspace{2cm} \ge \int_\M \nabla |\omega |. \nabla | \eta|\, dx + \int_M  V(x) |\omega(x)|_x |\eta(x)|_x\, dx
\end{eqnarray*}
and one checks the ideal property.  The details are easy and one concludes from the domination criteria mentioned above that the semigroup associated with $\tilde{\Delta} + {\mathcal R}$ is dominated by the semigroup associated with $\Delta+ V$. This applies in particular to ${\mathcal R} = Ricci$ and $V(x) = \rho(x)$ and gives  \eqref{Hodge-domination} under the sole assumption that $\rho^-$ satisfies \eqref{eq-form-moins}. 

\section{Intertwining  via the Riesz transform and applications}\label{sec4}

A fondamental tool in our analysis  will be the Riesz transform $R = d \Dtp^{-\frac{1}{2}}$. Recall  that $\Dtp = \Delta$ if $\M$ has infinite volume (in which case $0$ is not an eigenvalue of $\Delta$) and $\Dtp = \Delta(I- \Pi)$ if $\M$ has finite volume, $\Pi$ is then the projection on the constants.  Recall also that $L^2_{\overline{d}}(\Lambda^1)$ denotes the closure in $L^2(\Lambda^1)$ of the space of exact forms. 

Since the Laplacian is not necessarily invertible on $L^2(\M)$ one starts by  defining  the Riesz transform on the range 
$Im(\Dtp^\frac{1}{2})$ of $\Dtp^\frac{1}{2}$, 
\[
R : \begin{array}{c c c }  
     Im(\Dtp^\frac{1}{2}) &\to &  L^2_{\overline{d}}(\Lambda^1) \\
     f &\mapsto& d\Dtp^{-\frac{1}{2}}f.
     \end{array}
\]
We state some elementary properties.
\begin{lem}\label{lem321}  The Riesz transform $R$ extends to an operator (still denoted by $R$)
 \[
R :    \overline{Im(\Dtp^\frac{1}{2})} \to   L^2_{\overline{d}}(\Lambda^1)
 \]
which is an isometry and surjective. 
\end{lem}
\begin{proof}
First 
\beq
Im(\Dtp^\frac{1}{2}):=\left\{
f= \Delta^\frac{1}{2} g,  g \in H^1(\M) \textrm{ with } g\perp1 \textrm{ if the measure is finite}
\right\}.
\eeq 
Thus, if $f \in Im(\Dtp^\frac{1}{2})$, $f= \Dtp^\frac{1}{2} g$, then
\[
Rf= d g \in L^2_{\overline{d}}(\Lambda^1).
\]
Moreover, since $g \in H^1(\M)$, 
\begin{eqnarray}
\| Rf \|_2^2 &=& \langle  d\Dtp^{-\frac{1}{2}} f, d\Dtp^{-\frac{1}{2}} f \rangle \nonumber\\
                                         &=& \langle  \Dtp^{\frac{1}{2}}\Dtp^{-\frac{1}{2}} f,\Dtp^{\frac{1}{2}}\Dtp^{-\frac{1}{2}} f \rangle \nonumber\\
                                         &=&  \| f \|_2^2. \label{R=2}
\end{eqnarray}

Now we turn to the surjectivity in $L^2_{\overline{d}}(\Lambda^1)$.
Let $\o\in L^2_{\overline{d}}(\Lambda^1)$ be such that  $\o=dg $ for some $g\in H^1(\M)$. If the measure is finite, we may add a constant to $g$ so that  $g\perp1$. Thus, $g\in \D(\Dtp^{\frac{1}{2}})$ and 
$\Dtp^{\frac{1}{2}} g \in Im( \Dtp^{\frac{1}{2}})$.  Hence, $R (\Dtp^{\frac{1}{2}} g) = dg = \o$. 
For general $\o\in L^2_{\overline{d}}(\Lambda^1)$, there exists a sequence  $(g_n) \in H^1(\M)$ such that
$(d g_n)$ converges to $\o$ in $L^2(\Lambda^1)$.  As before, by adding a constant to $g_n$ we can assume $g_n \in \D(\Dtp^{\frac{1}{2}})$. So $d g_n = d \Dtp^{-\frac{1}{2}} \Dtp^{\frac{1}{2}} g_n = R (\Dtp^{\frac{1}{2}} g_n)$. It follows from the equality \eqref{R=2} that 
$\Dtp^{\frac{1}{2}} g_n$ is a Cauchy sequence and hence it converges for the $L^2$-norm to some $h \in \overline{Im(\Dtp^\frac{1}{2})}$. Thus, $ dg_n = R (\Dtp^{\frac{1}{2}} g_n)$ converges to $\o$ and to $R(h)$. This gives $\o = R(h)$ and proves the surjectivity. 
\end{proof}

Since $0$ is not an eigenvalue of $\Delta$ when $\M$ has infinite volume and it is not an eigenvalue of $\Delta(I-\Pi)$ when $\M$ has finite volume it follows by duality that 
$\overline{Im(\Dtp^\frac{1}{2})}  = L^2$ if $\M$ has infinite volume and $\overline{Im(\Dtp^\frac{1}{2})} = L^2_\perp$ if $\M$ has finite volume. 
To make the notation easier, we sometimes write  $L^2_\perp = L^2$ even when $\M$ has infinite volume to avoid discussing separately the cases of finite and infinite volume. 

We mention the following simple lemma. 
\begin{lem}\label{lem331}
 The adjoint operator $R^*$, 
 \[
R^* :   L^2_{\overline{d}}(\Lambda^1)   \to  L^2_\perp
 \] is defined  for $\o \in L^2_{\bar d}$ by
 \[
 \langle R^* \o, f \rangle_{L^2(\M)} =  \langle \o, R f \rangle_{L^2(\Lambda_1)} \textrm{ for all } f \in  L^2_\perp.
 \]
It satisfies
  \[
  R^* R = I \ \ {\rm on }\ L^2_\perp
  \]
  and since $R$ is surjective,
\[
  RR^* = I \ \ {\rm on }\ L^2_{\overline{d}}(\Lambda^1).
  \]
\end{lem}
Seen as an operator on $L^2(\Lambda^1)$, the operator $RR^*$ is the projection onto $L^2_{\overline{d}}(\Lambda^1)$. The operator $I- RR^*$ is called the Leray-Helmoltz projection. It plays an important role in the theory of Navier-Stokes equations. \\

Formally, it is clear that 
\[
R\, \Dtp R^* = \Dtf_{\overline{d}}.
\]
Since we wish  to avoid domain questions here we first  state  and prove this equality at the level of semigroups. 

\begin{prop}\label{inter1}
We have for all $t > 0$ and $\o \in L^2_{\overline{d}}(\Lambda^1)$,
\[
R\, e^{-t\Dtp} R^*\o  = e^{-t  \Dtf_{\overline{d}}}\, \o.
\]
\end{prop}
\begin{proof} We start with $\o = df$ and $f \in C_c^\infty(\M)$. If $\M$ has finite volume, we add a constant $c$ to $f$ so that $\int_{\M} f = 0$. In particular, $f \in D(\Dtp)$ and $\o = df= d(f+c)$. Therefore,
\begin{eqnarray*}
R\, e^{-t\Dtp} R^*\o &=& R\, e^{-t\Dtp} \Dtp^{-\frac{1}{2}} d^* (df)\\
&=& R\, e^{-t\Dtp} \Dtp^{-\frac{1}{2}} \Delta f\\
&=& R\, e^{-t\Dtp} \Dtp^{\frac{1}{2}}  f\\
&=& R\, \Dtp^{\frac{1}{2}} e^{-t\Dtp} f.
\end{eqnarray*}
Since $\Dtp^{\frac{1}{2}} e^{-t\Dtp} f \in Im(\Dtp^{\frac{1}{2}})$, we can write $R\, \Dtp^{\frac{1}{2}} e^{-t\Dtp} f = d \Dtp^{-\frac{1}{2}} \Dtp^{\frac{1}{2}} e^{-t\Dtp} f$. It follows from this and Proposition \ref{prop2.1} that 
\[
R\, e^{-t\Dtp} R^*\o = d e^{-t\Dtp} f = e^{-t  \Dtf_{\overline{d}}}\, \o.
\]
We extend this equality by a density argument to all $\o \in L^2_{\overline{d}}(\Lambda^1)$.
\end{proof}

It follows from the  proposition  that 
$$ D( \Dtf_{\overline{d}})  = \{ \omega \in L^2_{\overline{d}}(\Lambda^1), R^* \omega \in D(\Dtp) \}, \quad R\, \Dtp R^* \omega = \Dtf_{\overline{d}} \omega.$$
We can also write $D(\Dtp)$ in terms of  $D( \Dtf_{\overline{d}})$ by changing  $R^*$ into $R$. 

This intertwining property  imply  equality for the spectra  and essential spectra of the two involved operators. 

\begin{cor}\label{sigma=}
We have
\[ \sigma (\Dtp) = \sigma (\Dtf_{\overline{d}}),  \quad \sigma_{ess} (\Dtp) = \sigma_{ess} (\Dtf_{\overline{d}})
\]
\end{cor}
\begin{proof}
The equality of the spectra follows immediately from the intertwining property of the operators. For the essential spectra, it is enough to notice that every Weyl  sequence for one operator is transformed by $R^*$ or $R$ into a Weyl sequence for the other operator. This uses the fact that $R^*$ and $R$ are isometries between appropriate spaces (cf. Lemmas \ref{lem321} and  \ref{lem331}). 
 \end{proof}

The first part of Proposition \ref{prop1} can be obtained from the previous corollary. We restate the result in the next proposition. Recall that $\rho(x)$ is the smallest eigenvalue of the Ricci tensor at $x$. 
\begin{prop}\label{prop3.1.1}
Let $\M$ be a complete Riemannian manifold. Suppose  that  $\rho ^-$ is form bounded in the sense of \eqref{formbdd} for some $\alpha \in [0, 1)$. 
\begin{itemize}
\item If $vol({\mathcal M}) = \infty$, then $s(\Delta) \ge s(\Delta + \rho) \ge (1-\alpha) s(\Delta+ \rho^+)$. In particular, if $\rho \ge 0$ then $s(\Delta) = s(\Delta + \rho )$.
\item If  $vol({\mathcal M}) < \infty$, then $\lambda_1(\Delta) \ge s(\Delta + \rho) \ge (1-\alpha)s(\Delta+\rho^+)$. 
\end{itemize}
\end{prop}

\begin{proof}
  Applying Corollary \ref{sigma=} we obtain  the equality of the spectral bounds, that is
\[ 
s(\Dtp) = s(\Dtf_{\overline{d}}).
\]
Remember that $s(\Dtp)$ is either the spectral bound $s(\Delta)$  (if $vol(\M) =\infty$) or the spectral gap $\lambda_1(\Delta)$ (if $vol(\M) < \infty$). Concerning the right hand side term of the previous equality, we have from \eqref{dw=0}, 
\begin{eqnarray*}
s(\Dtf_{\overline{d}})&=& \inf\left\{ \int_\M | d^* \o |^2\, dx, \, \o \in  L^2_{\overline{d}}(\Lambda^1),  d^*\omega \in L^2, \|\o\|_2 = 1 \right\}\\
&=& \inf\left\{ \int_\M | d^* \o |^2\, dx + \int_\M | d \o |^2\, dx, \, \o \in  L^2_{\overline{d}}(\Lambda^1),  d^*\omega \in L^2, \|\o\|_2 = 1 \right\}\\
&\ge& \inf\left\{ \int_\M | d^* \o |^2\, dx + \int_\M | d \o |^2\, dx, \, \o \in H^1(\Lambda^1),  \|\o\|_2 = 1 \right\}\\
&=& s(\Dtf).
\end{eqnarray*}
On the other hand, it follows from the domination \eqref{Hodge-domination} that
\[ 
s(\Dtf) \ge s(\Delta + \rho).
\]
It follows that
\begin{equation}\label{egalite0}
s(\Dtp) \ge s(\Delta + \rho).
\end{equation}
The remaining parts in the proposition follow easily from this inequality. 
\end{proof}

\section{The spectral bound of a Schr\"odinger operator}\label{sec3}

We consider now  a Schr\"odinger operator $\Delta + V$ with a non-negative and locally bounded potential $V$. In many situations, the spectral bound of $\Delta$ satisfies  $s(\Delta) = 0$  and one may hope for $s(\Delta + V) > 0$ for appropriate $V$. This later property is of interest since it gives an exponential decay of the corresponding semigroup. We shall see that it also gives a lower bound for the spectral gap of $\Delta$.

In the Euclidean setting $\M = \R^n$, it is clear that if $V$ has compact support then $s(\Delta + V) = 0$. This suggests that in order to have $s(\Delta + V) > 0$ the potential should not vanish on some "large" subsets. It was proved in \cite{AB} that for bounded $V$, a necessary and suffisant condition for $s(\Delta + V) > 0$  is that for some radius $R > 0$, the integral of $V$ on balls with radius $R$ is strictly positive. This was extended to general manifolds in \cite{Ouhabaz01} where the condition becomes that the average on balls  is strictly positive, that is the condition \eqref{meanR} for some $R > 0$. Shortly after \cite{Ouhabaz01}, it was proved in \cite{Shen} that one might improve slightly the assumption that $V$ is bounded. In this section we reproduce and improve some arguments from \cite{Ouhabaz01} and \cite{Shen} in order to have suitable lower bounds for $s(\Delta + V)$. 

Suppose that  $\M$ satisfies the covering property $(A1)$. Thus,  $\M = \cup_i B(p_i,R)$ for some sequence $(p_i)_i \in \M$ and each $x \in \M$ is contained in at most $N_R$ of these balls. In particular, for every non-negative function $v$ on $\M$,
\beq\label{eq:covering}
   \sum_i  \int_{B(p_i,R)} v\, dx  \geq \int_\M v\, dx \geq  \frac{1}{N_R}  \sum_i  \int_{B(p_i,R)} v\, dx.
\eeq
Hence,  for $f \in H^1(\M)$,
\begin{eqnarray*}
  \langle (\Delta + V ) f, f\rangle  &=& \int_\M |\nabla f|^2 dx  + \int_\M  V f^2 dx  \\
  &\geq&   \frac{1}{N_R} \left(  \sum_i  \int_{B(p_i,R)} |\nabla f|^2 dx  + \int_{B(p_i,R)}  V f^2 dx\right)\\
  &\geq & \frac{1}{N_R}      \sum_i   s\left((\Delta + V)_{|B(p_i,R)} \right) \int_{B(p_i,R)}    f^2 dx\\
  &\geq&  \frac{1}{N_R}  \inf_i  s\left((\Delta + V)_{|B(p_i,R)} \right)    \sum_i  \int_{B(p_i,R)}    f^2 dx\\
  &\geq&  \frac{1}{N_R}  \inf_i  s\left((\Delta + V)_{|B(p_i,R)} \right)      \int_{\M}    f^2 dx.\\
\end{eqnarray*}
Here and in the sequel, $(\Delta + V)_{|B(p_i,R)}$  denotes the Schr\"odinger operator on the space $L^2(B(p_i,R))$ and subject to  Neumann boundary conditions. If $V$ is not present, $\Delta_{|B(p_i,R)}$ is the Neumann-Laplacian. It follows from the previous estimates that 
\begin{equation}\label{eq.min s}
s(\Delta + V) \ge\,  \sup_{R > 0} \frac{1}{N_R} \inf_i s\left((\Delta + V)_{|B(p_i,R)} \right).
\end{equation}
Our next task is  to estimate from below $s\left((\Delta + V)_{|B(p,R)} \right)$ independently of the point $p \in \M$. In order to do so we follow an argument in \cite{Shen} by reproducing the proof of the well known Fefferman-Phong lemma. Set $B = B(p,R)$, $| B | = vol(B)$ and denote by $u_B$ the average of a given $u \in H^1(B)$.  A simple calculation shows the equality
\begin{equation*}
\frac{1}{2 |B|} \iint_{B\times B} | u(x) - u(y) |^2 \, dx dy = \int_{B} | u(x) - u_B |^2 \, dx.
\end{equation*} 
Note that the first eigenvalue $\lambda_0$ of $\Delta_{|B}$ is of course $0$ and it is simple so that the second eigenvalue  $\lambda_1$ is $ > 0$.  Note that $\lambda_1 = \lambda_1(p,R)$ may a priori depend on $p$ and $R$. However, if $\M$ satisfies the Poincar\'e inequality \eqref{poincare} then $\lambda_1(p,R) \ge \kappa_R$. Therefore,  one may replace in the estimates below $\lambda_1(p,R)$ by 
$\kappa_R$. We do not use  this  right now  and keep going on with $\lambda_1(p,R)$.  

The variational inequality gives 
$$ \int_{B} | u(x) - u_B |^2 \, dx \le \frac{1}{\lambda_1(p,R)} \int_B | \nabla u(x) |^2 \, dx.$$
Hence,
\begin{equation*}
\frac{1}{2 |B|} \iint_{B\times B} | u(x) - u(y) |^2 \, dx dy \le \frac{1}{\lambda_1(p,R)} \int_B | \nabla u(x) |^2 \, dx.
\end{equation*} 
Hence
\begin{eqnarray*}
&&\hspace{-1cm} \frac{1}{ |B|} \iint_{B\times B}\left[  \frac{\lambda_1(p,R)}{2} | u(x) - u(y) |^2  + V(y) |u(y)|^2 \right] \, dx dy \\ 
&& \hspace{6cm}  \le  \int_B \left[ | \nabla u(x) |^2  + V(x) |u(x) |^2 \right] \, dx.
\end{eqnarray*} 
Next, we develop the square $| u(x) - u(y) |^2$ and use the inequality 
\[ 
2 \frac{\lambda_1(p,R)}{2} u(x) u(y) \le \frac{\lambda_1(p,R)^2}{4} \frac{|u(x)|^2}{ \frac{\lambda_1(p,R)}{2} + V(y)} + |u(y)|^2 (\frac{\lambda_1(p,R)}{2} + V(y))
\]
to obtain 
\begin{equation}\label{truc1}
\frac{\lambda_1(p,R)}{2} \fint_B\frac{V(y)}{ \frac{\lambda_1(p,R)}{2} + V(y)}\,dy \int_B |u(x)|^2\, dx \le \int_B \left[ | \nabla u(x) |^2  + V(x) |u(x) |^2 \right] \, dx.
\end{equation}
It follows from  this latter  inequality that
\begin{eqnarray}
s\left((\Delta + V)_{|B(p,R)} \right) &\ge&  \frac{ \lambda_1(p,R) }{\lambda_1(p,R)  + 2 \| V \|_{L^\infty(B)}} \frac{1}{|B|} \int_B V(y)\, dy  \\ \label{mino1}
&\ge&  \frac{ \lambda_1(p,R) \delta(R) }{\lambda_1(p,R)  + 2 \| V \|_{L^\infty(B)}}.\label{mino2} 
\end{eqnarray}
Here 
$$\delta(R) := \inf_{p \in \M} \fint_{B(p,R)} V(x)\, dx.$$
Finally, using \eqref{eq.min s} we obtain a lower bound for $s(\Delta + V)$ which we state  in the following theorem.

\begin{thm}\label{thm3.1}
Suppose that $\M$ satisfies $(A1)$ and  
let $V$ be a non-negative, locally bounded  potential. Then 
\begin{equation}\label{mino3}
s(\Delta + V) \ge \,  \sup_{R > 0} \inf_i \left[ \frac{1}{N_R} \frac{ \lambda_1(p_i,R)  }{\lambda_1(p_i,R)  + 2 \| V \|_{L^\infty(B(p_i,R))}}\right] \delta(R).
\end{equation}
In particular, if $V \in L^\infty (\M)$ and satisfies the average condition with $\delta(R) > 0$ for some $R > 0$ and $\M$
satisfies the Poincar\'e inequality  \eqref{poincare} with constant  $\kappa_R$,  then
\begin{equation}\label{mino4}
s(\Delta + V) \ge \, \sup_{R > 0}\,   \left[ \frac{1}{N_R} \frac{ \kappa_R}{\kappa_R + 2  \| V \|_{L^\infty(\M)}}\right]  \delta(R). 
\end{equation}
\end{thm}

\begin{rmk}\label{rk4.2}
 \begin{itemize}
 \item Using the fact  that $\frac{ab}{a+b} \ge \frac{1}{2} \min (a,b)$ for positive numbers $a, b$, 
 it follows from \eqref{truc1} that
 \begin{equation}\label{mino1-2}
 s\left((\Delta + V)_{|B(p,R)} \right) \ge \frac{1}{ 2}\fint_B \min \left( \frac{\lambda_1(p,R)}{2}, V(y) \right)\, dy.
 \end{equation}
 This estimate can be improved  as follows: if $k=\inf_{x \in K} V(x)$ on a domain $K$ of $\M$, then 
\beq\label{eq:s-rhogeqk}
s\left((\Delta + V)_{|_K}\right) \geq  k + \frac{1}{2} \fint_K \min\left(\frac{\lambda_1(K)}{2} , V(y)-k \right) dy.
\eeq
 \item Obviously, if $V$ is not bounded but  for some positive constant $w_0$, $\inf \{V, w_0\}$ satisfies the average condition for some $R > 0$, then we obtain the previous lower bound for $s(\Delta + V)$ from $s(\Delta + V ) \ge s(\Delta + \inf\{V,w_0\})$.
\item Suppose that the negative part $V^-$ of $V$ is non-trivial and it satisfies the form boundedness condition \eqref{formbdd} for some $\alpha \in [0,1)$. Then obviously
\begin{equation}\label{v-moins}
s(\Delta + V) \ge (1-\alpha) s(\Delta + V^+).
\end{equation}
Therefore, by applying the previous theorem to $\Delta + V^+$ one obtains the same lower bound with the factor $(1-\alpha)$  for $s(\Delta + V)$. In particular, if $V^+$ satisfies the average condition 
\[ \inf_{p \in \M} \fint_{B(p,R)} V^+\, dx  > 0,\]
then $s(\Delta + V) > 0$. 
\end{itemize}
\end{rmk}
\begin{rmk}\label{rk4.3}
Let $K$ be a domain of $\M$. 
It is  possible to obtain an estimate  for the bottom of the spectrum $s((\Delta+V)_{|_K})$ for  potentials  $V$ having negative part.
Set  $\alpha= \frac{\lambda_1(K)}{2}$ and $\beta=\beta(y)=V(y)$. Then 
\begin{align*}
& \fint |\nabla u(y)|^2 + V(y) u(y)^2 dy\\
\geq & \fint_K\fint_K  \left[ \alpha |u(x)-u(y)|^2   +  \beta u(y)^2 \right] dx dy\\
=& \fint_K\fint_K  \left[ \alpha |u(x)|^2   - 2 \alpha u(x) u(y) +   (\alpha+\beta) u(y)^2  \right] dx dy\\
=& \fint_K\fint_K  \left[ \left(\alpha -\frac{\alpha^2} {\alpha+ \beta} \right) |u(x)|^2   +\left( \frac{\alpha} {\sqrt{\alpha+ \beta}} u(x)- \sqrt{\alpha+ \beta} u(y) \right)^2 \right] dx dy\\
\geq & \left(\fint_K  \frac{\alpha \beta} {\alpha+ \beta} dy \right)  \; \fint |u(x)|^2 dx
\end{align*}
provided that for all $y\in K$, $\alpha + \beta(y)>0$, that is
\[
V(y)> - \frac{\lambda_1(K)}{2}.
\]
In the case, $\alpha>0, \beta \geq 0$, one has
\[
 0< \frac{1}{2} \min(\alpha,\beta) \leq \frac{\alpha \beta} {\alpha+ \beta} \leq \min(\alpha,\beta).
\] 
Whereas in the case $\beta<0$, if for some $\kappa \in (0,1)$,  $\alpha + \beta \geq \kappa \alpha$,
\[
 \frac{\beta}{\kappa} \leq \frac{\alpha \beta} {\alpha+ \beta} \leq \beta <0.
\] 
That is, if 
 \beq\label{eq:newclassV-}
  V^-(y)\leq  (1-\kappa) \frac{\lambda_1(K)}{2}, \ y \in K
  \eeq 
  then
 \[
 0> \frac{\frac{\lambda_1(K)}{2} V(y)}{\frac{\lambda_1(K)}{2}+  V(y)} \geq \frac{V(y)}{\kappa}.
 \]
  Thus,  we obtain for  $V$ satisfying  $\eqref{eq:newclassV-}$  for some $\kappa=\kappa(K)\in (0,1)$
 \beq\label{eq:new-s}
 s\left( (\Delta+ V)_{|_K} \right) \geq \frac{1}{2}  \fint_K \min\left(\frac{\lambda_1(K)}{2} ,V^+(y) \right)  -\frac{V^-(y)}{\kappa(K)} dy.
 \eeq
 
 Moreover, for any $k\geq 0$, if there exists $\kappa(K) \in (0,1)$ such that for all $y\in K$,
 \[
 (V(y)-k)^- \leq (1-\kappa(K)) \frac{\lambda_1(K)}{2},
 \]
 then it follows from  \eqref{eq:s-rhogeqk}and \eqref{eq:new-s}  that 
 \beq\label{eq:new-s-k}
 s\left( (\Delta+ V)_{|_K} \right) \geq k+ \frac{1}{2} \fint_K \left[ \min\left(\frac{\lambda_1(K)}{2} ,(V-k)^+(y) \right)  -\frac{(V-k)^-(y)}{\kappa(K)} \right] dy.
 \eeq
 \end{rmk}

We  combine these results and those in the previous section to obtain lower estimates for $s(\Dtp)$ which means lower estimates for the spectral bound $s(\Delta)$ (when $\M$ has infinite volume) and for the spectral gap $\lambda_1(\Delta)$ (when $\M$ has finite volume). Our bounds are given in terms of $\rho(x)$, the smallest eigenvalue of the Ricci tensor at the point $x$. 

\begin{thm}\label{thm-spectral}
Suppose that $\M$ satisfies the covering property  $(A1)$ for each $R > 0$. Suppose  that  $\rho ^-$ is form bounded in the sense of \eqref{formbdd} for some $\alpha \in [0, 1)$ and $\rho^+$ satisfies the average condition
\[ \delta(R) := \inf_{p \in \M} \fint_{B(p,R)} \rho^+(x)\, dx > 0.
\]

 Then 
\begin{equation}\label{mino3-rho}
s(\Dtp)  \ge \,  \sup_{R > 0} \inf_i \left[ \frac{1-\alpha}{N_R} \frac{ \lambda_1(p_i,R) }{\lambda_1(p_i,R)  + 2 \| \rho^+ \|_{L^\infty(B(p_i,R))}}\right]\delta(R).
\end{equation}
If  in addition $\rho^+  \in L^\infty (\M)$ and  $\M$
satisfies the Poincar\'e inequality  \eqref{poincare} with constant $\kappa_R$, then
\begin{equation}\label{mino4-rho}
s(\Dtp) \ge \, \sup_{R > 0}\,   \left[ \frac{1-\alpha}{N_R} \frac{ \kappa_R}{\kappa_R + 2  \| \rho^+ \|_{L^\infty(\M)}} \right]  \delta(R). 
\end{equation}
\end{thm}
\begin{proof} This is combination of \eqref{egalite0}, \eqref{v-moins} and Theorem \ref{thm3.1} with $V = \rho$.
\end{proof}

\section{A Bonnet-Myers type theorem}\label{sec5}

In this short section we give a proof for Theorems \ref{thm2} and \ref{thm1.6}.

\begin{proof}[Proof of Theorem \ref{thm2}]
 Suppose that  $\M$  satisfies $(A1)$, $(A2)$ and  the volume of balls has a sub-exponential  growth. We also suppose that $\inf \{ \rho ^+, w_0\}$ satisfies the average condition \eqref{meanR} for some $R > 0$ and some constant $w_0>0$ and  $\rho ^-$ satisfies \eqref{formbdd} for some $\alpha \in [0,1)$. If  $\M$ has infinite volume, then  $s(\Delta) = 0$ (actually, even the infimum of the essentiel spectrum is $0$, see \cite{Brooks}). On the other hand, by Proposition \ref{prop3.1.1} and Theorem \ref{thm3.1}
 \[
 s(\Delta) \ge (1-\alpha) s(\Delta + \rho^+) \ge (1-\alpha) s(\Delta + \inf\{\rho^+, w_0\}) > 0.
 \]
This is a contradiction with $s(\Delta) = 0$. Hence $\M$ has finite volume. 

Next, suppose that $\M$ is not compact and that for some $x_0 \in \M$
\begin{equation}\label{CGT}
Ricci(x) \ge - \frac{n}{n-1} \frac{1}{d(x,x_0)^2} 
\end{equation}
for $d(x,x_0)$ large enough. Then Theorem 4.9 (ii) from \cite{CGT} implies that $\M$ has infinite volume. We have seen already that this cannot be the case. Thus, $\M$ is compact. This proves the theorem.
\end{proof}

\begin{proof}[Proof of Theorem \ref{thm1.6}] We prove that  the volume of $\M$ is finite. The second assertion follows as in the previous proof. Note that, as already mentioned in the introduction,  the properties $(A1)$ and $(A2)$ are both satisfied since the manifold has Ricci curvature bounded from below.

Suppose for a contradiction that 
 $vol(M) = \infty$. Then  the two operators $\Delta_\perp$ and $\Delta$ are the same. In addition,  $\Delta$ is injective and hence  by duality 
$\overline{Im(\Delta^\frac{1}{2})} = L^2(\M)$. 

Recall that by Corollary \ref{sigma=}, 
\begin{equation}\label{eq-ess-spec}
\sigma_{ess}(\Delta)  = \sigma_{ess}(\Dtf_{\overline{d}})  \subseteq \sigma_{ess}(\Dtf).
\end{equation} 
Let $O$ be a smooth relatively compact open subset of $\M$ and denote by $\chi_O$ its  indicator function. We have 
 $$(1 + \Dtf)^{-1}  - (1 + \Dtf + Ricci^- \chi_O)^{-1} = (1 + \Dtf + Ricci^- \chi_O)^{-1} Ricci^- \chi_O(1 + \Dtf)^{-1}.$$
 The operator $\chi_O (1 + \Dtf)^{-1}$ maps into the Sobolev space $H^1(\Lambda^1)$ over the subset $O$. It follows from the Rellich compactness theorem that $\chi_O (1 + \Dtf)^{-1}$ is a compact operator on $L^2(\Lambda^1)$. Therefore, the difference $(1 + \Dtf)^{-1}  - (1 + \Dtf + Ricci^- \chi_O)^{-1}$ is a compact operator and hence  $\Dtf$ and $\Dtf + Ricci^- \chi_O$ have the same essential spectra. Applying  \eqref{eq-ess-spec} yields
 \begin{equation}\label{infsigmaess}
  \inf \sigma_{ess}(\Delta) \ge   \inf \sigma_{ess}(\Dtf)  =  \inf \sigma_{ess}(\Dtf + Ricci^- \chi_O) \ge s(\Dtf + Ricci^- \chi_O).
  \end{equation} 
  On the other hand, as we have seen in previous sections, $s(\Dtf + Ricci^- \chi_O) \ge s(\Delta + \rho^+ - \rho^- \chi_{\M\setminus O})$. 
Now we use the assumption that $Ricci$ is asymptotically non-negative. Given $\epsilon >0$ and let $O$ such that $\rho^- \chi_{\M\setminus O} < \epsilon$. Hence
\begin{eqnarray*}
 s(\Delta + \rho^+ - \rho^- \chi_{\M\setminus O}) &\ge&s(\Delta + \rho^+) - \epsilon\\
 &\ge& s(\Delta + \inf\{ \rho^+, w_0\} ) - \epsilon.
 \end{eqnarray*}
 Since $\inf\{ \rho^+, w_0\}$ satisfies the mean condition, we obtain from Theorem \ref{thm3.1} that $s(\Delta + \inf\{ \rho^+, w_0\} ) > 0$.  So $s(\Delta + \inf\{ \rho^+, w_0\} ) - \epsilon > 0$ for $\epsilon$ small. Inserting this into 
\eqref{infsigmaess} gives $ \inf \sigma_{ess}(\Delta) > 0$. In order to conclude, recall that the assumption on $Ricci^-$ implies that the volume has sub-exponential  growth (cf. \cite{Sturm}) and this together with the fact that the volume is infinite implies that  $ \inf \sigma_{ess}(\Delta) = 0$ (cf. \cite{Brooks}). This gives the contradiction and finishes the proof of the theorem. 
\end{proof}

As mentioned in the introduction we have made no attempt to obtain a diameter estimate for $\M$. Such estimates were obtained from $L^p$-type conditions on $\rho$ in \cite{Petersen-Sprouse} and \cite{aubry}  or from  a  Kato type property in \cite{Carron-Rose} and \cite{Rose}. We refer to these papers for additional information and references on related works. 

The first part of Theorem \ref{thm1.6} gives that $\M$ has finite measure and  the condition \eqref{CGT} is involved  in order to guaranty that $\M$ is compact. Once we have $\M$ has finite measure we may invoke  other conditions which allow to conclude  that $\M$ is compact. For example, $\M$ satisfies the log-Sobolev inequality \cite{Saloff2} or that the heat kernel and its inverse satisfy integrability conditions \cite{Gong-Wang}. 

 \begin{rmk} We have assumed in Theorem \ref{thm1.6} that $Ricci$ is asymptotically non-negative but we see from the proof that we can weaken this assumption in the following way. Since by Theorem \ref{thm3.1} $s(\Delta + \inf\{ \rho^+, w_0\} ) $  is larger than some  constant $\nu = \nu(R, \delta(R), \|\inf\{ \rho^+, w_0\} \|_\infty, N_R) > 0$ it is enough to have $\rho^-(x) < \nu$ outside a compact subset of $\M$. 
 \end{rmk}


\section{Weighted Riemannian manifolds}\label{sec6}

In the previous  parts of this paper we have chosen, for clarity of the exposition,  to deal  with the Laplace-Beltrami operator but our results extend easily with the same  ideas of proofs to the more general context of weighted Riemannian manifold. So we can consider Ornstein-Uhlenbeck type operators. 

Interesting examples arise already in the case of the Euclidean manifold $\R^n$ endowed with a probability  measure $d \mu = e^{-V} dx$. An important question is then to know whether the Poincar\'e inequality  holds and to obtain  some  meaningful  estimates for the spectral gap.    A typical example  is  the Ornstein-Uhlenbeck operator which is  the operator associated to  the Gaussian measure, i.e., $V(x) = \frac{|x|^2}{2}$. The potential is uniformly convex
and it  is the prototype of operators that satisfy the  $CD(k,\infty)$  Bakry-\'Emery dimension criterion.  Note that in this case, the $CD(k,\infty)$  condition is equivalent to  the fact that the Hessian  of $V$ is bounded from below in the quadratic form sense. We write this as $\Hess V (x)\geq k I_n$ for all $x\in \R^n$.  


In the case where $k > 0$, the Poincar\'e  inequality holds (and also the log-Sobolev) and  the spectral gap is bounded from below by $k$. See \cite{Bakry-Gentil-Ledoux}, Proposition 4.8.1.

In the next section we provide estimates of the spectral gap  of perturbations  of the exponential power distribution
$\frac{1}{Z_\alpha} e^{-\frac{|x|^\alpha}{\alpha}}$. For $\alpha \not=2$, these measures are  log-concave but not uniformly log-concave and the Bakry-\'Emery argument fails. Indeed, they only satisfy the $CD(0,\infty)$ criterion.
In connection with this we mention the  famous KLS conjecture which asserts the existence of a positive universal lower bound  (and in particular independent of the dimension) for the spectral gap of any (isotropic) log-concave measure on $\R^n$. We refer to \cite{A-GB} and also to \cite{YChen} for some recent progress.

We now describe the setting and notations. We denote again by $\M$ a complete Riemannian manifold  and we introduce a measure $d\mu= e^{-V} dx$  with a smooth potential $V$. Here, $dx$ stands as before for the standard Riemannian volume measure.  We still denote by  $d$ the exterior derivative on differential forms but here we denote $d_\mu^*$ its formal adjoint with respect to the measure $\mu$. One can then define  as previously  the Hodge Laplacian operators  (also called Witten Laplacians). 
The only difference in the construction is that in the definitions of the quadratic forms $\fra,\frb, \overset{\rightarrow}{\fra}$ and  $\overset{\rightarrow}{\fra_1}$ one has to replace the Riemannian measure $dx$ by the measure $d\mu$.  The Laplace-Beltrami on functions is $\Delta_\mu = d_\mu^* d$,  the Hodge Laplacian on forms is 
$\Dtf_\mu= d_\mu^* d  + d d_\mu^*$ and its part  on the closure of exact forms is $ \Dtf_{\mu,\overline{d}}=d d_\mu^*$. 

Assumption $(A1)$ is only a metric property and does not change in this setting. However, $(A2)$ means  that we require the Poincar\'e inequality with respect  to the measure $\mu$. Concerning the important Bochner's formula, it reads now as 
  $\Dtf_\mu = \tilde{\Delta}_\mu + \mathcal R$  where $\mathcal R$ is the generalized Ricci tensor given by 
  $ {\mathcal R} = Ricci + \Hess V$
   and 
$\tilde{\Delta}_\mu = \nabla^* \nabla  + \nabla_{\nabla V} $ is the rough Laplacian (here $\nabla$ is the Levi-Civita connexion). For all this, we refer to \cite{Coulhon-Devyver-Sikora}, Appendix A. 

In this setting, the domination property is formulated as follows. 
\begin{prop}\label{domination-mu}
Let $\rho(x)$ be the smallest eigenvalue of $\mathcal{R}(x) = Ricci(x) +\Hess V(x)$. Suppose that $\rho^-$ is form-bounded in the sense of \eqref{eq-form-moins}, that is  there exist $\alpha \in [0, 1)$ and $c_\alpha \in \R$ such that
\[  \int_{\M} \rho^- |u|^2\, d\mu \le \alpha \left[ \int_{\M} \nabla u. \nabla v\, d\mu + \int_{\M} \rho^+ u v \, d\mu \right] + c_\alpha \int_{\M} |u|^2\, d\mu
\]
for all $u \in L^2(\M, d\mu)$ such that $\int_M | \nabla u |^2\, d\mu + \int_\M \rho^+ |u|^2\, d\mu < \infty$. 
Then for $\omega \in L^2(\Lambda^1, d\mu)$ we have 
\begin{equation}\label{Hodge-domination-mu}
| e^{-t\Dtf_\mu} \omega (x) |_x \le e^{-t(\Delta_\mu + \rho)} |\omega(x)|_x.
\end{equation}
\end{prop}

\begin{proof} We follow the same arguments as for  \eqref{Hodge-domination}. One first establishes the  domination for the semigroups
of $\tilde{\Delta}_\mu$ and $\Delta_\mu$. This is stated in  \cite{Coulhon-Devyver-Sikora} (Theorem A.2). The perturbation by $\mathcal R$ can be handled exactly is in the unweighted case. 
\end{proof} 


Note also that $Ricci +\Hess V$ is precisely the tensor that appears in the  Bakry-\'Emery $\Gamma_2$ operator and that the  Bakry-\'Emery $CD(k,\infty)$ criterion holds if and only if $Ricci(x) + \Hess V(x) \geq k Id$ for all $x\in \M$ (see \cite{Bakry-Gentil-Ledoux}, Appendix C.6.3).
In this section, we work in the same spirit as in the unweighted case of the previous sections and in particular, we do not assume a uniform lower bound for $Ricci(x) + \Hess V(x)$. We can reproduce the results there with similar proofs. We  restate here Proposition \ref{prop3.1.1}  and Theorem \ref{thm-spectral} in the context of weighted manifolds. 
\begin{prop}\label{prop3.1.1-w}
Let $\rho$ be the smallest eigenvalue of the tensor ${\mathcal R}(x) = Ricci(x) + HessV(x)$ and  suppose  that  $\rho ^-$ is form bounded in the sense of \eqref{formbdd} for some $\alpha \in [0, 1)$. 
\begin{itemize}
\item If $vol_\mu({\mathcal M}) = \infty$, then $s(\Delta_\mu) \ge s(\Delta_\mu + \rho) \ge (1-\alpha) s(\Delta_\mu+ \rho^+)$. In particular, if $\rho \ge 0$, then $s(\Delta_\mu) = s(\Delta_\mu + \rho )$.
\item If  $vol_\mu({\mathcal M}) < \infty$, then $\lambda_1(\Delta_\mu) \ge s(\Delta_\mu + \rho) \ge (1-\alpha)s(\Delta_\mu+\rho^+)$. 
\end{itemize}
\end{prop}

\begin{thm}\label{thm-spectral-mu}
Suppose that $\M$ satisfies the covering property  $(A1)$ for each $R > 0$. Let $\rho(x)$ be the smallest eigenvalue of $Ricci(x) + \Hess V(x)$. We assume that $\rho ^-$ is form bounded in the sense of \eqref{formbdd} for some $\alpha \in [0, 1)$ and $\rho^+$ satisfies the average condition
\begin{equation}\label{ave-mu}
 \delta(R) := \inf_{p \in \M} \fint_{B(p,R)} \rho^+(x)\, d\mu(x) > 0.
\end{equation}
1) Suppose that $\mu(\M) = \infty$. Then
\begin{equation}\label{mino3-rho-mu}
s(\Delta_\mu)  \ge \,  \sup_{R > 0} \inf_i \left[ \frac{1-\alpha}{N_R} \frac{ \lambda_1(p_i,R) }{\lambda_1(p_i,R)  + 2 \| \rho^+ \|_{L^\infty(B(p_i,R))}}\right]\delta(R).
\end{equation}
If  in addition $\rho^+  \in L^\infty (\M, \mu)$ and  $\M$
satisfies the Poincar\'e inequality  \eqref{poincare} for the measure $\mu$, then
\begin{equation}\label{mino4-rho-mu}
s(\Delta_\mu) \ge \, \sup_{R > 0}\,   \left[ \frac{1-\alpha}{N_R} \frac{ \kappa_R}{\kappa_R + 2  \| \rho^+ \|_{L^\infty(\M)}} \right]  \delta(R). 
\end{equation}
2) If $\mu(\M) < \infty$, then $\Delta_\mu$ has a spectral gap $\lambda_1(\mu)$ and $\lambda_1(\mu)$ satisfies the same lower bounds as  in \eqref{mino3-rho-mu} and \eqref{mino4-rho-mu}. 
\end{thm}

The result on the compactness of the manifold is reserved to the unweighted Laplace-Beltrami operator but the first part of Theorem \ref{thm2} still holds in the weighted case. 

\begin{thm}\label{thm2-mu}
Let $\M$ and $\mu = e^{-V} dx$ be a before and  $\rho(x)$ be  the smallest eigenvalue of $Ricci(x) + \Hess V(x)$. 
Suppose that the manifold  $\M$  satisfies $(A1)$ and  $(A2)$ for the measure $\mu$.  Assume also that  the $\mu$-volume of balls has a sub-exponential  growth and  that $\inf \{ \rho ^+, w_0\}$ satisfies the average condition \eqref{ave-mu} for some $R > 0$ and some constant $w_0>0$. Suppose that $\rho ^-$ satisfies \eqref{formbdd} for some $\alpha \in [0,1)$. 
Then $\M$ has finite $\mu$-volume. 
\end{thm}
The proof is the same as for Theorem \ref{thm2}. Note that in the weighted case, the result of Brooks on the bottom of the spectrum still holds with the same proof as in \cite{Brooks}. This means that if the $\mu$-volume of balls has a sub-exponential  growth then $s(\Delta_\mu) = 0$. 
 
 Related results to the previous theorem are known under different geometric conditions. We refer to \cite{Xue Mei Li} and \cite{Gong-Wang}. 


\section{Examples: perturbations of radial measures on $\R^n$}\label{sec7}
In this section we provide applications of the above theory in the case of a probability measure on $\R^n$.  We obtain here spectral gap estimates for some perturbations of  exponential power distributions on $\R^n$. The density with respect to the Lebesgue measure of these distributions is given by 
\[
d\mu_\alpha=\frac{1}{Z_\alpha}e^{-\frac{|x|^\alpha}{\alpha}} dx
\]
for $\alpha >0$, where $Z_\alpha$ is a renormalisation constant such that $\mu_\alpha$ is a probability measure.

It is well known that this measure admits a spectral gap if and only if $\alpha\geq 1$.
Except for the case $\alpha=2$ where it corresponds  to the Gaussian measure, it does not satisfy the Bakry-\'Emery criterion. This is due to the lack of strict convexity of the potential $V_\alpha=\frac{|x|^\alpha}{\alpha}$  near zero in the case $\alpha>2$ and at infinity
 for $1 <\alpha<2$.  
 
  For $\alpha\geq 1$, the measure $\mu_\alpha$ is radial and log-concave. In this case,  sharp estimates of the spectral gap $\lambda_1(\mu_\alpha)$ are known (see also Theorem \ref{thm:bobkov-radial} below).

We state two results on lower bounds for the spectral gap for (not necessarily radial) perturbations of the exponential power distribution. 

%
\begin{prop}\label{prop:ex1}
Let $\alpha \geq 2$,  $a>0$ and let $R_a=(a n) ^{ \frac{1}{\alpha}}$.
Let $d\mu=e^{-V} dx$ be   a probability measure on $\R^n$  whose potential V satisfies 
\[
 \left\{ \begin{array} {c c c} 
V(x)=  \frac{|x|^\alpha}{\alpha} & \textrm{ if } & |x| \leq  R_a\\
Hess V(x) \geq c \,  (a n)^{1-\frac{2}{\alpha}} & \textrm{ if } & |x| \geq R_a\\
 \end{array}\right.
\]
for some $c \le 1$. Then there exists a constant $C$, independent of the dimension $n$, such that 
\[
  \lambda_1(\mu)  \ge C \left\{ \begin{array} {c c c} 
   \min \left(\frac{1}{4 a},\frac{1}{2}, c \right) (an)^{1-\frac{2}{\alpha}}     & \textrm{ if } & 0<a\leq 1\\
   \min \left(\frac{1}{4} ,c a^{1-\frac{2}{\alpha}}  \right) n^{1-\frac{2}{\alpha}}   & \textrm{ if } & a>1. \\
 \end{array}\right.
\]
\end{prop}


The next result concerns the case  $1<\alpha\leq 2$. 
\begin{prop}\label{prop:ex2}
Let $1<\alpha \leq 2$, $a>0$ and $R_a= (an) ^{ \frac{1}{\alpha}}$.
Let $d\mu=e^{-V} dx$ be   a  probability measure on $\R^n$ whose potential $V$ satisfies 
\[
 \left\{ \begin{array} {c c c} 
Hess V(x) \geq c (\alpha-1) \,  (a n )^{1-\frac{2}{\alpha}} & \textrm{ if } & |x| \leq R_a\\
V(x)=  \frac{|x|^\alpha}{\alpha} & \textrm{ if } & |x| \geq R_a\\
 \end{array}\right.
\]
for some $c \le 1$. Then there exists a constant $C$, independent of the dimension $n$ such that
\[
  \lambda_1(\mu)  \ge C \left\{ \begin{array} {c c c} 
  \min\left(\frac{1}{4a}, \frac{(\alpha-1)}{2}, c (\alpha-1) \right) (an)^{1-\frac{2}{\alpha}}    & \textrm{ if } & a\geq 1 \\
   \min \left(\frac{1}{4} ,\frac{(\alpha-1)}{2}, c a^{1-\frac{2}{\alpha}} (\alpha-1)  \right) n^{1-\frac{2}{\alpha}}     & \textrm{ if } &0< a<1.\\
 \end{array}\right.
\]
\end{prop}
The reason to state the propositions with the constant $a$ is to consider measures which are not necessarily close in total variation to the exponential power distribution. We refer to \cite{Milman} for stability results  of the spectral gap for log-concave measures with respect to the total variation distance.

The proofs are based on the  arguments developed in the previous sections. However, instead of taking arbitrary balls in the covering of the manifold as we did before,  we will choose  a   covering which is  more suitable  to the geometry of the potential $V$.

Recall  from the previous sections that for any covering $(K_i)_{i}$ of $\R^n$ one has
\beq \label{eq:covering-1}
\lambda_1(\mu) \geq \frac{1}{N} \inf_i  s\left((L + \rho)_{|_{K_i}}\right).
\eeq
with $N$ is the intersection  number of the covering. Recall  also that if $U \geq k>0$ for some constant $k>0$ on $K$, then 
\beq\label{eq:s-rho-cst}
s\left((\Delta_\mu + U)_{|_K}\right) \geq k
\eeq
whereas  if only $U\geq 0$ we have as in \eqref{mino1-2}
\beq\label{eq:s-rhogeq0}
s\left((\Delta_\mu+ U_{|_K}\right) \geq \frac{1}{2} \frac{1}{\mu(K)}\int_K \min\left(\frac{\lambda_1(K)}{2} , U(x)\right) d\mu(x).
\eeq
 
In order to prove  the last two propositions let us recall the following estimate for the spectral gap of radial measures (see \cite{Bobkov} and   \cite{Bonnefont-Joulin-Ma}).
\begin{thm}\label{thm:bobkov-radial}
Let $\mu$ be a radial log-concave probability measure on $\R^n$, then $\mu$ admits a spectral gap $\lambda_1(\mu)$ and 
\[
\frac{n-1}{\int_{\R^n} |x|^2 d\mu(x)}  \leq \lambda_1(\mu) \leq \frac{n}{\int_{\R^n} |x|^2 d\mu(x)}. 
\]
\end{thm}

It is important to point out that this   theorem remains valid (without any modification) if $\R^n$ is replaced by a ball $B(0,R)$ for some $R>0$ or by its complement $B(0,R)^c$. We will use the above result in these later situations. 

Note also that in the case  of dimension 1, the lower bound in the above estimate is not good, but in this case the spectral gap is bounded  from below by $\frac{1}{12}\frac{1}{\int_{\R} |x|^2 d\mu(x)}$ (see \cite{bobkov99}).

 \begin{proof}[Proof of Proposition \ref{prop:ex1}]
 By Proposition \ref{prop3.1.1-w} we have 
 \[
 \lambda_1(\mu) := \lambda_1(\Delta_\mu) \ge s(\Delta_\mu + \rho).
 \]
 We estimate from below the spectral bound $s(\Delta_\mu + \rho)$.

 The result for $a\geq 1$ follows from the case  $a=1$ by replacing the constant  $c$  by $\min (c a^{1-\frac{\alpha}{2}},1)$). Hence we  only consider the case $0<a\leq 1$.
 
 By replacing  $a$ by $\frac{a n}{n-\alpha}$ we may rewrite $R_a$ as $R_a=  (a(n-\alpha)) ^{ \frac{1}{\alpha}}$. 
 We first recall that for $|x|\leq R_a $ 
 \[
 Hess V(x)= |x|^{\alpha-2} \left( Id + (\alpha-2) \frac{x x^T}{|x|^2} \right)
\]
and thus $\rho(x)= |x|^{\alpha-2}$ for $|x|\leq R_a $.

We  take the following  simple covering (with covering number $N = 1$) 
\[
\R^n= B(0,R_a) \cup B(0,R_a)^c = B_a \cup B_a^c.
\]
We estimate separately $s\left((\Delta_\mu+\rho)_{|_{B_a^c}}\right)$ and $s\left((\Delta_\mu+\rho)_{|_{B_a}}\right)$.  The first one is very easy since $\rho$ is uniformly bounded from below and so by  \eqref{eq:s-rho-cst},
\beq\label{eq:est-L1Bc}
s\left((\Delta_\mu+\rho)_{B_a^c}\right) \geq c ( a(n-\alpha))^{1-\frac{2}{\alpha}}.
\eeq
Next we estimate $s\left((L+\rho)_{B_a}\right)$. Since the potential vanishes at 0, we use \eqref{eq:s-rhogeq0}.
It is radial, log-concave on $B_a$ and  proportional to  the exponential power distribution on $B_a$. Thus, by Theorem \ref{thm:bobkov-radial}, 
\beq\label{eq:est-bobkov-L1B}
\lambda_1(B_a) \geq \frac{n-1}{\fint_{B_a} |x|^2 d\mu_\alpha(x)} \geq \frac{n-1}{R_a^2}. 
\eeq
Moreover, for $x\in B_a$ and  $0<a\leq 1$,
\[
\rho(x)= |x|^{\alpha-2} \leq R_a^{(\alpha-2)} \leq a  \frac{n-1}{R_a^2} \leq a {\lambda_1(B_a)}.
\]
Therefore, by \eqref{eq:s-rhogeq0},
\[
s\left((\Delta_\mu+\rho)_{B_a}\right) \geq  \frac{1}{2}  \fint_{B_a} \min\left(\frac{1}{2a},1\right)  \rho(y) d\mu(y).
\]
We now provide an estimate  of this last quantity as $n$ tends to infinity. Set $I_{\alpha,n, R, \gamma}:=\int _{r=0}^{R}r^{\gamma} r^{n-1} e^{- \frac{r^\alpha }{\alpha}}dr$, one has
\[\fint_{B_a}  \rho(y) d\mu(y)=\frac{I_{\alpha,n, R_a, \alpha-2}}{I_{\alpha,n, R_a, 0}}.\]

By two  changes of variables ($y=\frac{r^\alpha}{\alpha}$ and  $y =\frac{n-\alpha}{\alpha} u$), one easily obtains
\begin{eqnarray*}
I_{\alpha,n, R, \gamma} &=&  \alpha^{\frac{\gamma}{\alpha} + \frac{n-\alpha}{\alpha}} \int_0^{\frac{R^\alpha}{\alpha}} y^{\frac{\gamma}{\alpha} + \frac{n-\alpha}{\alpha}}
e^{-y}dy\\
&=&  \frac{1}{\alpha} (n-\alpha)^{\frac{\gamma}{\alpha} + \frac{n}{\alpha}}  \int_0^{\frac{R^\alpha}{n-\alpha}} \exp\left(- \left(\frac{n-\alpha}{\alpha}\right) (u-\ln u) \right) u^{\frac{\gamma}{\alpha}} du.
\end{eqnarray*}
Hence
\[
I_{\alpha,n, R_a, \gamma}= \frac{1}{\alpha}(n-\alpha)^{\frac{\gamma}{\alpha} + \frac{n}{\alpha}} \int_0^{a} \exp\left(- \left(\frac{n-\alpha}{\alpha}\right)\psi(u) \right) f_\gamma (u)du.
\]
with $\psi(u)=u-\ln(u)$ and $f_\gamma(u)= u^{\frac{\gamma}{\alpha}} $. The minimum of $\psi$ is attained for $u=1$ and using the standard Laplace method with 
$J_{\alpha,n,a,\gamma}= \int_0^{a} \exp\left(- \left(\frac{n-\alpha}{\alpha}\right)\psi(u) \right) f_\gamma (u)du$, one has 
\[
\frac{J_{\alpha,n, a, \alpha-2}}{J_{\alpha,n, a, 0}} \sim_{n \to +\infty}  
\left\{ \begin{array}{c c c }
\frac{f_{\alpha-2}(1)}{f_0(1)} = 1 & \textrm{ if } a\geq 1\\
\frac{f_{\alpha-2}(a)}{f_0(a)} = a^{1-\frac{2}{\alpha}} & \textrm{ if } 0<a<1.\\
\end{array}\right.
\]
Therefore,  
\beq\label{eq:est-int-rho}
\fint_{B_a}  \rho(y) d\mu(y) \sim \left\{ \begin{array}{c c c }
  ( a n)^{1-\frac{2}{\alpha}}\sim R_a^{\alpha-2} & \textrm{ if }& 0<a<1\\
   n^{1-\frac{2}{\alpha}}\sim R_1^{\alpha-2} & \textrm{ if }& a\geq 1.\\
\end{array}\right.
\eeq
This implies a lower bound for  $s\left((\Delta_\mu+\rho)|_{B_a}\right)$ as follows
\beq\label{eq:est-L1B}
s\left((\Delta_\mu+\rho)|_{B_a}\right) \gtrsim \left\{ \begin{array}{c c c }
 \min\left(\frac{1}{4a},\frac{1}{2}\right)  ( a (n-\alpha))^{1-\frac{2}{\alpha}} & \textrm{ if }& 0<a<1\\
 \frac{1}{4} ((n-\alpha))^{1-\frac{2}{\alpha}} & \textrm{ if }& a=1.\\
\end{array}\right.
\eeq
The implicit constant in this lower bound is independent of $n$. 
Combining  \eqref{eq:est-L1B} and \eqref{eq:est-L1Bc} finishes  the proof of the proposition. 
\end{proof}

The proof of Proposition \ref{prop:ex2} is similar,  we only have to change  the roles of $B_a$ and $B_a^c$.
\begin{proof}[Proof of Proposition \ref{prop:ex2}]
 The result for $0<a <1$ will follow from the one for $a= 1$ by replacing  $c$ by  $\min \left(c a^{1-\frac{2}{\alpha}},1\right)$. We consider  only  the cases $ a\geq 1$.
 We set $R_a=  (a(n-\alpha)) ^{ \frac{1}{\alpha}}$.
 We first recall that for $|x|\geq R_a $ 
 \[
 Hess V(x)= |x|^{\alpha-2} \left( Id + (\alpha-2) \frac{x x^T}{|x|^2} \right)
\]
and thus $\rho(x)=(\alpha-1) |x|^{\alpha-2}$ for $|x|\geq R_a $. 

We take again the  covering
\[
\R^n= B(0,R_a) \cup B(0,R_a)^c = B_a \cup B_a^c.
\]
The estimate of $s\left((\Delta_\mu+\rho)_{|_{B_a}}\right)$ is obvious  since $\rho$ is uniformly bounded from below on 
$B_a$ and hence
\beq\label{eq:estimate-B2}
s\left((L+\rho)|_{B_a}\right) \geq c (\alpha-1) (a(n-\alpha))^{1-\frac{2}{\alpha}}.
\eeq

To estimate $s\left((\Delta_\mu+\rho)_{|_{B_a^c}}\right)$, we first use Theorem \ref{thm:bobkov-radial} and  obtain 
\[
 \lambda_1(B_a^c) \sim \frac{n}{\fint_{B_a^c} |x|^2 e^{-\frac{|x|^\alpha}{\alpha}} dx}= \frac{n\tilde I_{\alpha,n, R_a, 0}}{\tilde I_{\alpha,n, R_a, 2}}
\]
with $
\tilde I_{\alpha,n, R, \gamma}:=\int_R^\infty r^\gamma r^{n-1}e^{-\frac{r^\alpha}{\alpha}} dr$. Note that 
\[
\tilde I_{\alpha,n, R_a, \gamma}= \frac{1}{\alpha} (n-\alpha)^{\frac{\gamma}{\alpha} + \frac{n}{\alpha}} \int_{a}^\infty \exp\left(- \left(\frac{n-\alpha}{\alpha}\right)\psi(u) \right) f_\gamma (u)du.
\]
The Laplace method  gives 
\[
\fint_{B_a^c} |x|^2 e^{-\frac{|x|^\alpha}{\alpha}} dx \sim_{n\to \infty}
\left\{ \begin{array}{c c c }
 ( a n)^{\frac{2}{\alpha}}\sim R_a^2 & \textrm{ if }& a\geq 1\\
  (n)^{\frac{2}{\alpha}} \sim R_1^2& \textrm{ if }& 0<a\leq 1.\\
\end{array}\right.
\]
Hence
\[
 \lambda_1(B_a^c) \sim \left\{ \begin{array}{c c c }
  \frac{ n^{1-\frac{2}{\alpha}} }  {a^{\frac{2}{\alpha}}} \sim \frac{n}{R_a^2}& \textrm{ if }& a\geq 1\\
    n^{1-\frac{2}{\alpha}} \sim \frac{n}{R_1^2}& \textrm{ if }& a\leq  1.\\
\end{array}\right.
\]
Thus for $|x|\geq R_a$ and $a\geq 1$
\[
\rho(x) \leq \frac{\alpha-1}{R_a^{2-\alpha} } \lesssim a(\alpha-1) \lambda_1(B_a^c).
\]
By \eqref{eq:s-rhogeq0},

\[
s\left((\Delta_\mu+\rho)_{|_{B_a^c}}\right) \gtrsim \frac{1}{2} \min\left(\frac{1}{2a},(\alpha-1)\right) \fint_{B_a^c} |x|^{\alpha-2} d\mu_\alpha(x).
\]
As before, 
\[
\fint_{B_a^c} |x|^{\alpha-2} d\mu_\alpha(x)=  \frac{\tilde I_{\alpha,n, R_a, \alpha-2}}{\tilde I_{\alpha,n, R_a, 0}} \sim
 \left\{ \begin{array}{c c c }
 {(a n)^{1-\frac{2}{\alpha}} }  & \textrm{ if }& a\geq 1\\
    n^{1-\frac{2}{\alpha}} & \textrm{ if }& 0<a\leq 1\\
\end{array}\right.
\]
and the result follows.
\end{proof}

\begin{rmk}
For the particular case of the exponential power distribution $\mu_\alpha$  for $\alpha > 1$   Proposition \eqref{prop:ex1} and Proposition\eqref{prop:ex2} (now $c=1$) give the good asymptotic for the spectral gap $\lambda_1(\mu_\alpha)$ up to a factor $\frac{1}{4}$ in the case $\alpha \geq 2$ 
and   a factor $\min\left(\frac{1}{4},\alpha-1 \right)$ in the case $1<\alpha \leq 2$.

\end{rmk}

\end{document}